\documentclass[12pt,twoside]{article}

\usepackage{amssymb}
\usepackage{amsthm}
\usepackage{latexsym}
\usepackage{verbatim}

\newtheorem{theorem}{Theorem}
\newtheorem{definition}[theorem]{Definition}
\newtheorem{proposition}[theorem]{Proposition}

\newtheorem{lemma}[theorem]{Lemma}

\newcommand{\I}{{\mathbf{I}}}

\begin{document}

\noindent
\textbf{Proc.\ Ninth EMI Conf.\ (2023) pp.\ 49--58}

\noindent
\textbf{DOI 10.5281/zenodo.10428478}

\vspace{15mm}

\begin{center}
{\Large \textbf{ON THE CONSTRUCTION} \\

\bigskip

\textbf{OF COHN'S UNIVERSAL LOCALIZATION}}

\bigskip

\textbf{JOHN A. BEACHY}

\bigskip

\textbf{Abstract}

\end{center}

{\small 

\baselineskip=12pt

\begin{quote}
For an associative ring $R$ 
we investigate a construction 
of Cohn's universal ring of fractions $R_{\Sigma}$,
defined relative to a multiplicative set $\Sigma$ of matrices.
The construction of $R_{\Sigma}$
avoids the Ore condition,
which is necessary to construct a ring of fractions
relative to a multiplicative set of elements of $R$.
But a similar condition,
which we call the ``pseudo-Ore'' condition,
plays an important role in the construction of $R_{\Sigma}$.
We show that this condition in fact determines 
the equivalence relation used in the construction of $R_{\Sigma}$,
and provides information about left $R_{\Sigma}$-modules.
\end{quote} }

\vspace{2mm}

\baselineskip=14pt

Throughout this paper,
$R$ will denote an associative ring with identity $1$
that is not necessarily commutative,
and $X$ will denote a unital left $R$-module.
The major new results are given in 
Theorems~\ref{eleven} and \ref{twelve}.
The first concerns the equivalence relation used to define $R_{\Sigma}$,
and the second characterizes the kernel of the canonical mapping
from $X$ to $R_{\Sigma} \otimes_R X$.

A subset $S \subset R$ is said to be a multiplicative set
if it contains $1$ and is closed under multiplication.
Correspondingly, a set $\Sigma$ of square matrices 
is said to be multiplicative if
it contains all permutation matrices over $R$,
is closed under multiplication (when defined), and
if $C,D \in \Sigma$, then
$\left[ \begin{array}{cc} C & A \\ 0 & D \end{array} \right] \in \Sigma$
for any matrix $A$ over $R$ of the appropriate size.
We denote the set of $n \times n$ matrices in $\Sigma$ by $\Sigma_n$.

It is shown by Cohn in \cite{COHN85} that 
given a multiplicative set $\Sigma$ of matrices over $R$,
there exists a ring $R_{\Sigma}$ and a ring homomorphism
$\lambda : R \rightarrow R_{\Sigma}$
that is universal with respect to inverting the matrices in $\Sigma$.
The ring $R_{\Sigma}$ is constructed 
by adjoining enough elements to invert the given matrices,
subject to the necessary relations.
As pointed out in \cite{COHN85},
this provides little information about $\ker (\lambda)$,
and to address this, 
other constructions have been given by
Malcolmson (see \cite{MALCOLMSON82})
and Gerasimov (see \cite{GERASIMOV82}).
Malcolmson's construction has been simplified 
by the present author in \cite{BEACHY93}.
It is that construction which will be investigated here,
as a way to provide further information 
about the kernels of the canonical mappings
$\lambda : R \rightarrow R_{\Sigma}$
and $\mu_X : X \rightarrow R_{\Sigma} \otimes_R X$.

\newpage

\topmargin=.0in
\pagestyle{myheadings}
\markboth{John A Beachy}{Universal localization}
\setcounter{page}{50}

The construction in \cite{BEACHY93} proceeds as follows.
Cohn has shown that each element of $R_{\Sigma}$ 
is an entry $e_i \lambda (C)^{-1} e_j^t$,
in a matrix of the form $\lambda (C)^{-1}$,
where $C \in \Sigma$, $e_i, e_j$ are unit row vectors,
and $e_j^t$ denotes the transpose of $e_j$.
Addition of triples requires us to model elements of the form 
$\lambda (a) \lambda (C)^{-1} \lambda (b)^t$,
where $a,b \in R^n$, and $C \in \Sigma_n$.
Since it is just as easy to construct a module of quotients,
we consider ordered triples $(a,C,x^t)$,
where $a \in R^n$, $C \in \Sigma_n$, and $x \in X^n$,
where $X$ is any unital left $R$-module.
Following Malcolmson's development in \cite{MALCOLMSON82},
we first define an addition on ordered triples.

\begin{definition} %\label{one}
The sum of ordered triples $(a,C,x^t )$, $(b,D,y^t )$,
with $a \in R^n$, $C \in \Sigma_n$, $x \in X^n$ 
and $b \in R^m$, $D \in \Sigma_m$, $y \in X^m$ 
is defined by
$$
(a,C,x^t ) + (b,D,y^t ) 
= \left( [a \;\;\; b] ,
\left[ \begin{array}{cc} C & 0 \\ 0 & D \end{array} \right] ,
\left[ \begin{array}{r} x^t \\ y^t \end{array} \right] \right) \, .
$$
\end{definition}
 
The next step is to introduce the following equivalence relation
(see Definition~2.1 of \cite{BEACHY93})
under which the equivalence classes of ordered triples
form a commutative semigroup.

\begin{definition}  %\label{two} 
Let $a,b \in R^n$, 
$C_1, C_2 \in {\Sigma}_n$, 
and $x, y \in X^n$.
If there exist invertible $n \times n$ matrices $U_1, U_2$ over $R$
such that $a = bU_1$, $y^t = U_2 x^t$,
and $C_2 U_1 = U_2 C_1$, 
then we write
$$
(a,C_1,x^t) \equiv (b,C_2,y^t) \, ,
$$
and we say that 
$(a,C_1,x^t)$ and $(b,C_2,y^t)$ are \emph{congruent} via $U_1, U_2$.
\end{definition}

\begin{lemma}[\cite{BEACHY93}]  % \label{three}
Under the congruence relation $\equiv$, 
addition of triples is commutative.
\end{lemma}

\begin{proof}
If $C \in \Sigma_n$ and $D \in \Sigma_m$, then
$$
(a,C,x^t ) + (b,D,y^t ) \equiv (b,D,y^t ) + (a,C,x^t )
$$
since
$$
[ a \;\; b ] = 
\begin{array}{cc} 
 [ b & a ] \\ 
     &    
\end{array}
\left[ \begin{array}{ll} 
     0   & \I_m \\ 
     \I_n & 0 
\end{array} \right] \; ,
$$
$$
\left[ \begin{array}{cc} 
      D & 0 \\ 
      0 & C  
\end{array} \right]
\left[ \begin{array}{ll} 
     0   & \I_m \\ 
     \I_n & 0 
\end{array} \right] 
=
\left[ \begin{array}{ll} 
     0   & \I_m \\ 
     \I_n & 0 
\end{array} \right]
\left[ \begin{array}{cc} 
      C & 0 \\ 
      0 & D
\end{array} \right] \; ,
$$
and 
$$
\left[ \begin{array}{c} 
     y^t \\ 
     x^t  
\end{array} \right]
=
\left[ \begin{array}{ll} 
     0   & \I_m \\ 
     \I_n & 0 
\end{array} \right]
\left[ \begin{array}{c} 
     x^t \\ 
     y^t  
\end{array} \right] \; ,
$$
where $\I_n$ and $\I_m$ are identity matrices
of the appropriate sizes.
\end{proof}

\newpage

We denote by $\Sigma^{-1} X_0$ the subsemigroup
generated by all triples of the form $(0,C,x^t)$ or $(b,D,0^t)$,
for $a \in R^n$, $C \in \Sigma_n$, $x \in X^n$ and all $n > 0$.
Since addition is commutative,
the elements of $\Sigma^{-1} X_0$ can be put in the form
$(0,C,x^t) + (b,D,0^t)$.

\begin{definition}  % \label{four}
Let $a \in R^n$, $C \in \Sigma_n$, $x \in X^n$ 
and $b \in R^m$, $D \in \Sigma_m$, $x \in X^m$.
If there exist $z_1, z_2 \in \Sigma^{-1} X_0$
such that $(a,C,x^t) + z_1 \equiv (b,D,y^t) + z_2$,
then we write
$$
(a,C,x^t) \sim (b,D,y^t) \, .
$$ 

\noindent
The equivalence classes of ordered triples 
under the equivalence relation $\sim$
will be denoted by $[a:C:x^t]$,
and $\Sigma^{-1} X$ will denote
the set of all such equivalence classes.
\end{definition}

Proposition~2.3 of \cite{BEACHY93} shows that $\sim$
defines a congruence on the semigroup of ordered triples,
and that $\Sigma^{-1} X$ is an abelian group.
With an appropriate multiplication,
it is then shown in \cite{BEACHY93} that $\Sigma^{-1} R$ 
is a ring isomorphic to the universal localization $R_{\Sigma}$,
and that $\Sigma^{-1} X$ is a left module over $R_{\Sigma}$
that is naturally isomorphic to 
$R_{\Sigma} \otimes_R X$.

To construct a ring of left fractions 
using a multiplicative set $S \subset R$,
we need to be able to replace any product $a_1 c_1^{-1}$
with a product $c_2^{-1} a_2$,
where $a_1, a_2 \in R$ and $c_1, c_2 \in S$
(see \cite{LAM98}).
This leads to the left Ore condition:
given $a_1 \in R$ and $c_1 \in S$,
there exist $a_2 \in R$ and $c_2 \in S$ 
such that $c_2 a_1 = a_2 c_1$.

Addition in $R_{\Sigma}$ does not require the left Ore condition.
However, by a fundamental result (see Lemma~2.4 of \cite{BEACHY93}), 
if $a \in R^m$, 
$C_1 \in {\Sigma}_n$, 
$C_2 \in {\Sigma}_m$, 
$x \in X^n$,
then
$$
(aA_1,C_1,x^t) \sim (a,C_2,A_2x^t)
$$
for any $m \times n$ matrices $A_1, A_2$ over $R$
such that $C_2 A_1 = A_2 C_1$. 
This looks like the left Ore condition,
since in $R_{\Sigma}$ we are replacing
$\lambda (A_1) \lambda (C_1)^{-1}$ by $\lambda (C_2)^{-1} \lambda (A_2)$.
We will use this condition to define a new relation.
Note the important point that the matrices $C_1, C_2$
may have different sizes.

\begin{definition} % \label{five} 
Let $a \in R^n$, 
$C_1 \in {\Sigma}_n$, 
$x \in X^n$ and
$b \in R^m$, 
$C_2 \in {\Sigma}_m$, 
$y \in X^m$.
Suppose that there exist $m \times n$ matrices $A_1, A_2$ over $R$
and factorizations $a = bA_1$, $y^t = A_2 x^t$.
If $C_2 A_1 = A_2 C_1$, 
then we write
$(a,C_1,x^t) \geq (b,C_2,y^t)$.

In this case, we say that $(a,C_1,x^t) \geq (b,C_2,y^t)$
via $A_1, A_2$.
\end{definition}

\begin{lemma} 
\label{six}
The relation $\geq$ is reflexive, transitive, and respects addition.
\end{lemma}

\begin{proof}
Using the identity matrix,
it follows immediately that $\geq$ is reflexive.

To show that the transitive law holds, let
$(a_1, C_1, x_1^t) \geq (a_2, C_2, x_2^t)$ via $A_1, A_2$
and let $(a_2, C_2, x_2^t) \geq (a_3, C_3, x_3^t)$ via $B_2, B_3$.
Then $a_1 = a_2 A_1$, $x_2^t = A_2 x_1^t$, and $C_2 A_1 = A_2 C_1$
in the first case, and
$a_2 = a_3 B_2$, $x_3^t = B_3 x_2^t$, and $C_3 B_2 = B_3 C_2$ 
in the 

\newpage

\noindent
second case.  Substituting yields 
$a_1 = a_2 A_1 = a_3 (B_2 A_1)$,
$x_3^t = B_3 x_2^t = (B_3 A_2) x_1^t$,
and $C_3 (B_2 A_1) = (C_3 B_2) A_1 = (B_3 C_2) A_1 
= B_3 (C_2 A_1) = B_3 (A_2 C_1) = (B_3 A_2) C_1$.
It follows that $(a_1, C_1, x_1^t) \geq (a_3, C_3, x_3^t)$
via $B_2 A_1, B_3 C_2$,
showing that $\geq$ is a transitive relation.

To show that $\geq$ respects addition, suppose that
$(a_1, C_1, x_1^t) \geq (a_2, C_2, x_2^t)$ via $A_1, A_2$.
If $(a_3, C_3, x_3^t)$ is any ordered triple, then
$$
(a_1, C_1, x_1^t) + (a_3,C_3, x_3^t) 
\geq (a_2, C_2, x_2^t) + (a_3,C_3, x_3^t)
$$ via the matrices
$\left[ \begin{array}{cc} A_1 & 0 \\ 0 & \I \end{array} \right]$ and
$\left[ \begin{array}{cc} A_2 & 0 \\ 0 & \I \end{array} \right]$.
\end{proof}

The next proposition is Lemma~2.4 of \cite{BEACHY93}.
For the sake of completeness,
we include the proof here.

\begin{proposition}[Left pseudo-Ore condition] 
\label{seven}
Let $a \in R^n$, 
$C_1 \in {\Sigma}_n$, 
$x \in X^n$,
$b \in R^m$, 
$C_2 \in {\Sigma}_m$, 
$y \in X^m$.
If $(a,C_1,x^t) \geq (b,C_2,y^t)$,
then $[a:C_1:x^t] = [b:C_2:y^t]$.
\end{proposition}

\begin{proof}
We have 
\begin{eqnarray*}
[a \;\;\; aA_1]
= 
\begin{array}{lr} 
[a & 0] \\ & 
\end{array} 
\left[ \begin{array}{lc} 
\I_m &  A_1 \\ 
 0^t & \I_n 
\end{array} \right] 
& , &
\left[ \begin{array}{lr} 
\I_m & A_2 \\ 
  0  & \I_n 
\end{array} \right]
\left[ \begin{array}{r} 
0 \\ 
x^t 
\end{array} \right] 
= 
\left[ \begin{array}{c} 
A_2 x^t \\ 
    x^t 
\end{array} \right] 
\makebox[15mm]{and}
\end{eqnarray*}
%and
\begin{eqnarray*}
\left[ \begin{array}{cc} 
C_2 & 0    \\ 
0   & C_1 
\end{array} \right] 
\left[ \begin{array}{lc} 
\I_m & A_1 \\ 
0 & \I_n 
\end{array} \right] 
& = & 
\left[ \begin{array}{lc} 
\I_m &  A_2 \\ 
  0  & \I_n 
\end{array} \right] 
\left[ \begin{array}{cc} 
C_2 & 0 \\ 
 0  & C_1 
\end{array} \right] \, ,
\end{eqnarray*}
so by definition we have
\begin{eqnarray*}
\left( [a \;\;\; a A_1 ], 
\left[ \begin{array}{cc} C_2 & 0 \\ 0 & C_1 \end{array} \right] ,
\left[ \begin{array}{r} 0^t \\ x^t \end{array} \right] \right) 
& \equiv & 
\left( [a \;\;\; 0 ] ,
\left[ \begin{array}{cc} C_2 & 0 \\ 0 & C_1 \end{array} \right] ,
\left[ \begin{array}{c} A_2 x^t \\ x^t \end{array} \right] \right) 
\end{eqnarray*}
via $U_1 = 
\left[ \begin{array}{lc} 
\I_m &  A_1 \\ 
  0  & \I_n 
\end{array} \right] $ and
$U_2 =
\left[ \begin{array}{lr} 
\I_m & A_2 \\ 
  0  & \I_n 
\end{array} \right]$.
It follows from the definition of $\sim$
that $(a,C_2,0^t)$ and $(0,C_1,x^t)$
act as neutral elements for addition, and therefore
\begin{eqnarray*}
(aA_1,C_1,x^t )  & \sim & (a,C_2,0^t) + (aA_1,C_1,x^t)     \\
& = & \left( [a \;\;\; aA_1], 
\left[ \begin{array}{cc} C_2 & 0 \\ 0 & C_1 \end{array} \right] ,
\left[ \begin{array}{r} 0^t \\ x^t \end{array} \right] \right) \\
& \equiv & \left( [a \;\;\; 0] ,
\left[ \begin{array}{cc} C_2 & 0 \\ 0 & C_1 \end{array} \right] ,
\left[ \begin{array}{c} A_2 x^t \\ x^t \end{array} \right] \right)    \\
& = & (a,C_2,A_2x^t) + (0,C_1,x^t)                                      \\
& \sim & (a,C_2,A_2x^t)  \, .
\end{eqnarray*}
This completes the proof,
since, by definition, $\equiv$ implies $\sim$,
and $\sim$ is transitive.
\end{proof}

\newpage

\noindent
\textbf{Example 1}.
For any vectors $b,x$ and matrices $C,D \in \Sigma$
we have $(0,C,x^t) \geq (b,D,0^t)$ 
via the zero matrices of the appropriate size.
On the other hand, 
if $b \neq 0$ 
we cannot have $(b,D,0^t) \geq (0,C,x^t)$
since there does not exist a matrix $A_1$ with $b = 0 A_1$, 
and, similarly, if $x \neq 0$,
there does not exist a matrix $A_2$ with $x^t = A_2 0^t$. $\Box$

\bigskip

The above example shows that we need to introduce 
a right pseudo-Ore condition.
Note that because the matrices $U_1, U_2$ 
in the definition of the congruence relation $\equiv$ are invertible,
the relation $\equiv$ is in fact symmetric.

\begin{definition}  % \label{eight} 
Let $a \in R^n$, 
$C_1 \in {\Sigma}_n$, 
$x \in X^n$ and
$b \in R^m$, 
$C_2 \in {\Sigma}_m$, 
$y \in X^m$.
Suppose that there exist $n \times m$ matrices $A_1, A_2$ over $R$
and factorizations $x^t = A_1 y^t$, $b = aA_2$. 
If $C_1 A_2 = A_1 C_2$, 
then we write $(a,C_1,x^t) \leq (b,C_2,y^t)$.

In this case, we say that $(a,C_1,x^t) \leq (b,C_2,y^t)$
via $A_1, A_2$.
\end{definition}

The relation $\leq$ is reflexive, transitive, and respects addition
(the proofs are dual to those in Lemma~\ref{six}).
Part (b) of the next proposition establishes that
the right pseudo-Ore condition holds in $\Sigma^{-1} X$.

\begin{proposition}  %#09
\label{nine}
Let $a \in R^n$, 
$C_1 \in {\Sigma}_n$, 
$x \in X^n$,
$b \in R^m$, 
$C_2 \in {\Sigma}_m$, 
$y \in X^m$.

\smallskip

\emph{(a)}
We have $(a,C_1,x^t) \geq (b,C_2,y^t)$
if and only if 
$(b,C_2,y^t) \leq (a,C_1,x^t)$.

\smallskip

\emph{(b)}
If $(a,C_1,x^t) \leq (b,C_2,y^t)$,
then $[a:C_1:x^t] = [b:C_2:y^t]$.
\end{proposition}

\begin{proof}
(a)
By a careful application of the definitions,
$(a,C_1,x^t) \geq (b,C_2,y^t)$ via $A_1, A_2$
if and only if $a = bA_1$, $y^t = A_2 x^t$, and $C_2 A_1 = A_2 C_1$ 
and
$(b,C_2,y^t) \leq (a,C_1,x^t)$ via $B_1, B_2$
if and only if $y^t = B_1 x^t$, $a = bB_1$, and $C_2 B_2 = B_1 C_1$.
Thus $(a,C_1,x^t) \geq (b,C_2,y^t)$ via $A_1, A_2$
if and only if
$(b,C_2,y^t) \leq (a,C_1,x^t)$ via $A_2, A_1$.

(b)
This follows immediately from part (a) and Proposition~\ref{seven}.
\end{proof}

\begin{lemma}
\label{ten}
Let $a \in R^n$, $C \in {\Sigma}_n$, $x \in X^n$,
$b \in R^m$, $D \in {\Sigma}_m$, $y \in X^m$.
Then

\smallskip

\emph{(a)}
$(a,C,x^t) \geq (a,C,x^t) + (b,D,0^t)$ 
and 
$(a,C,x^t) + (0,D,y^t) \geq (a,C,x^t)$;

\smallskip

\emph{(b)}
$(a,C,x^t) + (b,D,0^t) \leq (a,C,x^t)$
and 
$(a,C,x^t) \leq (a,C,x^t) + (0,D,y^t)$. 
\end{lemma}

\begin{proof}
(a)
Since
$ \left[ \begin{array}{cc} C & 0 \\ 0 & D \end{array} \right]
\left[ \begin{array}{c} \I \\ 0 \end{array} \right] 
= \left[ \begin{array}{c} \I \\ 0 \end{array} \right] C$, 
by definition we have
\begin{eqnarray*}
(a,C,x^t) =
\left( \begin{array}{lr} [a & b] \\ & \end{array}
\left[ \begin{array}{c} \I \\ 0 \end{array} \right] , C , x^t \right)
& \geq &
\left( [a \;\;\; b] ,
\left[ \begin{array}{cc} C & 0 \\ 0 & D \end{array} \right] ,
\left[ \begin{array}{r} \I \\ 0 \end{array} \right] x^t \right) \\
& = & (a,C,x^t) + (b,D,0^t) \, .
\end{eqnarray*}

\newpage

Since 
$C \; [\I \;\;\; 0] 
= \begin{array}{lr} [\I & 0] \\ & \end{array}
\left[ \begin{array}{cc} C & 0 \\ 0 & D \end{array} \right]$, 
by definition we have
\begin{eqnarray*}
(a,C,x^t) + (0,D,y^t)
& = &
\left( a \; [\I \;\;\; 0] ,
\left[ \begin{array}{cc} C & 0 \\ 0 & D \end{array} \right] ,
\left[ \begin{array}{r} x^t \\ y^t \end{array} \right] \right) \\
& \geq &
\left( a,C, \begin{array}{lr} [\I & 0] \\ & \end{array}
\left[ \begin{array}{c} x^t \\ y^t \end{array} \right] \right)
= (a,C,x^t) \, .
\end{eqnarray*}

(b)
This follows immediately from part (a) and 
Proposition~\ref{nine}~(a).
\end{proof}

The following theorem shows that the left and right 
``pseudo-Ore'' conditions determine
the equivalence relation $\sim$
used in the construction of $\Sigma^{-1} X$.

\begin{theorem}
\label{eleven}
Let $n,m$ be positive integers, let
$a \in R^n$, $C \in {\Sigma}_n$, $x \in X^n$,
and let
$b \in R^m$, $D \in {\Sigma}_m$, $y \in X^m$.
The following conditions are equivalent:  %% in $\Sigma^{-1} X$:

\smallskip

\emph{(1)}
$(a,C,x^t ) \sim (b,D,y^t )$;

\smallskip

\emph{(2)}
there exist 
$u \in R^k$, 
$E, F \in \Sigma_k$, 
$z \in X^k$,
for some positive integer $k$, such that
%there exist triples $(0,E_1,e_1)$ and $(e_2,E_2,0)$ such that 
$(a,C,x^t) + (0,E,z^t) \geq (b,D,y^t) + (u,F,0^t)$;

\smallskip

\emph{(3)}
there exist triples $(a_1,C_1,x_1^t)$ and $(b_1,D_1,y_1^t)$ such that

\noindent
$(a,C,x^t ) \leq (a_1,C_1,x_1^t )$,
$(a_1,C_1,x_1^t ) \geq (b_1,D_1,y_1^t )$, and
$(b_1,D_1,y_1^t ) \leq (b,D,y^t )$.
\end{theorem}

\begin{proof}
(1) $\Rightarrow$ (2):
Suppose that $(a,C,x^t ) \sim (b,D,y^t )$.
Then by definition there exist triples
$(0,E,z^t)$, $(u_2,F_2,0^t)$,
$(0,E_2,z_2^t)$, $(u,F,0^t)$ such that
$$
(a,C,x^t) + (0,E,z^t) + (u_2,F_2,0^t)
\equiv
(b,D,y^t) + (0,E_2,z_2^t) + (u,F,0^t) \, .
$$
Since the relation $\equiv$ respects addition,
if $E \in \Sigma_j$ and $F \in \Sigma_k$ with $j < k$,
then we can add $k-j$ copies of $(0,1,0)$ 
to both $(0,E,z^t)$ and $(0,E_2,z_2^t)$ 
while maintaining the given identity.
A similar argument can be given if $j > k$,
so without loss of generality we can assume that $j = k$.

It follows from Lemma~\ref{ten}~(a) that
$$
(a,C,x^t) + (0,E,z^t) \geq 
(a,C,x^t) + (0,E,z^t) + (u_2,F_2,0^t) \,
$$
and that
$$
(b,D,y^t) + (0,E_2,z_2^t) + (u,F,0^t) \geq 
(b,D,y^t) + (u,F,0^t) \, .
$$
Since $\geq$ is transitive by Lemma~\ref{six}
and $\equiv$ implies $\geq$,
we have
$$
(a,C,x^t) + (0,E,z^t) \geq (b,D,y^t) + (u,F,0^t) \, .
$$

\newpage

(2) $\Rightarrow$ (3):
Given the triples $(0,E,z^t)$ and $(u,F,0^t)$ in condition (2),
let 
$$
(a_1,C_1,x_1^t) = (a,C,x^t) + (0,E,z^t)
$$
and 
$$
(b_1,D_1,y_1^t) = (b,D,y^t) + (u,F,0^t) \, .
$$
Then 
$(a,C,x^t ) \leq (a_1,C_1,x_1^t)$ by Lemma~\ref{ten}~(b),
$(a_1,C_1,x_1^t ) \geq (b_1,D_1,y_1^t )$ by hypothesis, and
$(b_1,D_1,y_1^t ) \leq (b,D,y^t )$ by Lemma~\ref{ten}~(a).

(3) $\Rightarrow$ (1):
This follows immediately from 
Propositions~\ref{nine} and \ref{seven}.
\end{proof}

In the following diagram, 
we denote $ \leq $ 
by showing the first triple below the second.
Then $(a,C,x^t) \sim (b,D,y^t)$ if and only if there
there exist triples $(a_1,C_1,x_1^t)$ and $(b_1,D_1,y_1^t)$ 
such that the following relationship holds.

\begin{picture}(320,90)
\put(50,0){\begin{picture}(200,90)
     \put(75,10){\begin{picture}(30,90)
	  \put(20,10){\makebox(0,0){$(a,C,x^t)$}}
     \end{picture}}
     \put(85,10){\begin{picture}(50,90)
          \put(10,30){\makebox(0,0){$\leq$}}
          \put(13,20){\vector(1,1){30}}
     \end{picture}}
     \put(115,10){\begin{picture}(30,70)
          \put(20,60){\makebox(0,0){$(a_1,C_1,x_1^t)$}}
     \end{picture}}
     \put(125,10){\begin{picture}(30,90)
          \put(10,50){\vector(1,-1){30}}
          \put(15,30){\makebox(0,0){$\geq$}}
     \end{picture}}
     \put(160,10){\begin{picture}(30,90)
	  \put(10,10){\makebox(0,0){$(b_1,D_1,x_1^t)$}}
     \end{picture}}
     \put(160,10){\begin{picture}(50,90)
          \put(35,30){\makebox(0,0){$\leq$}}
          \put(13,20){\vector(1,1){30}}
     \end{picture}}
     \put(195,0){\begin{picture}(30,30)
     \put(15,70){\makebox(0,0){$(b,D,y^t)$}}
     \end{picture}}
\end{picture}}
\end{picture}

Corollary~7.11.9 of \cite{COHN85} states if $r \in R$,
then $r \in \ker ( \lambda)$
if and only if 
for some $C, D \in \Sigma$
there is a relation of the form
$$
\left[ \begin{array}{cc} 
    0  &  r  \\
    0  &  0 
\end{array} \right] 
=
\left[ \begin{array}{cc} 
    A_{11}  &  A_{12}  \\
    C       &  A_{22} 
\end{array} \right] 
\left[ \begin{array}{cc} 
    B_{11}  &  B_{12}  \\
    D       &  B_{22} 
\end{array} \right]  \, .
$$
We call this Gerasimov's criterion,
since it has been developed in \cite{GERASIMOV82}.
Using Theorem~\ref{eleven},
we can extend it to modules.
We let $\mu_X$ denote the canonical mapping
$\mu_X : X \rightarrow \Sigma^{-1} X$
defined by setting $\mu_X (x) = [1:1:x]$, for all $x \in X$.

\begin{theorem}
\label{twelve}
The following conditions are equivalent for the inversive set $\Sigma$
and $x \in X$:

\smallskip

\emph{(1)}
$x \in \ker ( \mu_X )$, for the canonical mapping
$\mu_X : X \rightarrow \Sigma^{-1} X$;
%$(1,1,0) \sim (1,1,x)$;

\smallskip

\emph{(2)}
there exist 
$a \in R^n$,
$z \in X^n$,
and $C, D \in \Sigma_n$, 
for some $n > 0$, such that
$$
(1,1,0) + (0,D,z^t) \geq (1,1,x) + (a,C,0^t) \, ;
$$

%\smallskip

\emph{(3)}
there exist relations of the form
$$
\left[ \begin{array}{c} 
    x   \\
    0^t      
\end{array} \right] 
=
\left[ \begin{array}{cc} 
    a_{11}  &  a_{12}  \\
    C       &  A 
\end{array} \right] 
\left[ \begin{array}{c} 
    y^t   \\
    z^t      
\end{array} \right] 
\makebox[10mm]{and}
\left[ \begin{array}{cc} 
    a_{11}  &  a_{12}  \\
    C       &  A 
\end{array} \right] 
\left[ \begin{array}{c} 
    B   \\
    D      
\end{array} \right] 
=
\left[ \begin{array}{c} 
    0   \\
    0      
\end{array} \right] 
$$
for vectors $a_{11}, a_{12}$ over $R$,
$y, z$ over $X$,
and matrices $A, B, C, D$ such that $C, D \in \Sigma$.
%$$
%\left[ \begin{array}{cc} 
%    0  &  r  \\
%    0  &  0 
%\end{array} \right] 
%=
%\left[ \begin{array}{cc} 
%    A_{11}  &  A_{12}  \\
%    C       &  A_{22} 
%\end{array} \right] 
%\left[ \begin{array}{cc} 
%    B_{11}  &  B_{12}  \\
%    D       &  B_{22} 
%\end{array} \right] 
%$$
%with $C, D \in \Sigma$.
\end{theorem}

\newpage

\begin{proof}
(1) $\Rightarrow$ (2):
Since $x \in \ker ( \mu_X )$ if and only if $(1,1,0) \sim (1,1,x)$,
this is a direct application of Theorem~\ref{eleven}.

(2) $\Rightarrow$ (3):
Suppose that
$(1,1,0) + (0,D,z^t) \geq (1,1,x) + (a,C,0^t)$
via $A_1, A_2$.
Writing $A_1$ and $A_2$ in block form, there exist
$a_{11}, b_{11} \in R$,
$a_{12}, a_{21}, b_{12}, b_{21} \in R^n$, and
$A,B \in M_n (R)$ such that
$$
[1 \;\; 0] = 
\begin{array}{lr} [1 & a] \\ &  \end{array}
\left[ \begin{array}{cc} a_{11} & a_{12} \\ a_{21}^t & A \end{array} \right]
\makebox[8mm]{}
\left[ \begin{array}{cc} b_{11} & b_{12} \\ b_{21}^t & B \end{array} \right]
\left[ \begin{array}{c} 0 \\ z^t \end{array} \right] =
\left[ \begin{array}{c} x \\ 0^t \end{array} \right] 
$$
and
$$
\left[ \begin{array}{cc} 1 & 0 \\ 0 & C \end{array} \right]
\left[ \begin{array}{cc} a_{11} & a_{12} \\ a_{21}^t & A \end{array} \right]
=
\left[ \begin{array}{cc} b_{11} & b_{12} \\ b_{21}^t & B \end{array} \right]
\left[ \begin{array}{cc} 1 & 0 \\ 0 & D \end{array} \right] \, .
$$
This gives us the following equations:
\begin{center}
$\begin{array}{cccc}
\;\; a_{11} + a \cdot a_{21}^t = 1 \;\; &
\;\; a_{12} + aA = 0 \;\; &
\;\; b_{12} \cdot z^t = x \;\; &
\;\; B z^t = 0^t \;\; \\
&&& \\
a_{11} = b_{11} &
a_{12} = b_{12} D &
C a_{21}^t = b_{21}^t &
CA = BD \, .
\end{array}$
\end{center}
Substituting $a_{12} = b_{12}D$ 
in the equation $a_{12} + aA = 0$,
we only need to use the following equations
in order to obtain the desired result:
\begin{center}
$\begin{array}{cccc}
\;\; aA + b_{12}D = 0 \;\; &
\;\; CA - BD = 0 \;\; &
\;\; b_{12} \cdot z^t = x \;\; &
\;\; B z^t = 0^t \, .
\end{array}$
\end{center}
These equations show that 
$$
\left[ \begin{array}{cc} a & b_{12} \\ C & -B \end{array} \right]
\left[ \begin{array}{c} A \\ D \end{array} \right] 
=
\left[ \begin{array}{c} 0 \\ 0 \end{array} \right] 
\makebox[10mm]{and}
\left[ \begin{array}{cc} a & b_{12} \\ C & -B \end{array} \right]
\left[ \begin{array}{c} 0 \\ z^t \end{array} \right] 
=
\left[ \begin{array}{c} x \\ 0^t \end{array} \right] \, .
$$

(3) $\Rightarrow$ (1):
%$$
%\left[ \begin{array}{cc} 
%    0  &  r  \\
%    0  &  0 
%\end{array} \right] 
%=
%\left[ \begin{array}{cc} 
%    A_{11}  &  A_{12}  \\
%    C_{21}  &  A_{22} 
%\end{array} \right] 
%\left[ \begin{array}{cc} 
%    B_{11}  &  B_{12}  \\
%    D_{21}  &  B_{22} 
%\end{array} \right] 
%$$
%with $C_{21} \in \Sigma_n$
%and $ D_{21} \in \Sigma_m$.
Suppose that there are relations of the form
$$
\left[ \begin{array}{c} 
    x   \\
    0^t      
\end{array} \right] 
=
\left[ \begin{array}{cc} 
    a_{11}  &  a_{12}  \\
    C       &  A 
\end{array} \right] 
\left[ \begin{array}{c} 
    y^t   \\
    z^t      
\end{array} \right] 
\makebox[10mm]{and}
\left[ \begin{array}{cc} 
    a_{11}  &  a_{12}  \\
    C       &  A 
\end{array} \right] 
\left[ \begin{array}{c} 
    B   \\
    D      
\end{array} \right] 
=
\left[ \begin{array}{c} 
    0   \\
    0      
\end{array} \right] \, ,
$$
where 
$C \in \Sigma_n$. $D \in \Sigma_m$,
$A, B$ are $n \times m$ matrices over $R$,
$a_{11} \in R^n$, $a_{12} \in R^m$,
$y \in X^n$, and $z \in X^m$.
Then $CB = - AD$,
and so it follows from the left pseudo-Ore condition that
$$
(-a_{11}B,D,z^t) \sim (-a_{11},C,-Az^t) \, .
$$

Since $-a_{11}B = a_{12}D$, we have
$$
(-a_{11}B,D,z^t) =
(a_{12}D,\I_m D,z^t) \sim
(a_{12},\I_m,z^t) \, ,
$$

\newpage

\noindent
by an easy application of the left pseudo-Ore condition.
%for any vectors $a, b$ over $R$ of the correct size.
%Let $a = -e_1 A_{11}$ be the negative of the first row of $A_{11}$,
%and let $b^t = B_{22} e_k^t$ be the last column of $B_{22}$.
%Thus
%$$
%( - e_1 A_{11} B_{11} : D_{21} :  B_{22} e_k^t ) 
%\sim 
%( - e_1 A_{11} : C_{21} : - A_{22} B_{22} e_k^t ) \, .
%$$
%From the original relations we have 
%$- A_{11} B_{11} = A_{12} D_{21}$, 
%and so 
%$$ 
%( - e_1 A_{11} B_{11} : D_{21} :  B_{22} e_k^t ) 
%=
%( e_1 A_{12} D_{21} : D_{21} :  B_{22} e_k^t ) 
%\sim
%( e_1 A_{12} : \I_m :  B_{22} e_k^t ) \, .
%$$
Similarly, we have
$- A z^t = C y^t$, 
and so
$$
(a_{11},C,-Az^t)
=
(a_{11},C \I_n,Cy^t)
\sim
(a_{11}, \I_n,y^t) \, .
$$
Therefore the sum
$$
(a_{11},\I_n,y^t) + (a_{12},\I_m,z^t)
\sim
(a_{11},C,-Az^t) + (-a_{11},C,-Az^t)
\sim
(a_{11} - a_{11},C,-Az^t)
$$
must belong to $\Sigma^{-1} X_0$.
%Thus we have
%$$
%( - e_1 A_{11} : \I_n : B_{12} e_k^t )
%\sim
%( e_1 A_{12} : \I_m :  B_{22} e_k^t ) \, ,
%$$
%and therefore the sum
%$( e_1 A_{11} : \I_n : B_{12} e_k^t )
%+
%( e_1 A_{12} : \I_m :  B_{22} e_k^t )$ 
%belongs to $\Sigma^{-1} R_0$.

It follows easily from the left pseudo-Ore condition
that $(a,I_n,y^t) \geq (1,1,a \cdot y^t)$
and $(a,C,y^t) + (a,C,z^t) \geq (a,C,(y+z)^t)$,
for any $a \in R^n$, $C \in \Sigma_n$, 
and $y, z \in X^n$.
We conclude that 
$$
(1,1,x) 
= 
(1,1,a_{11} \cdot y^t + a_{12} \cdot z^t)
\sim
(1,1,a_{11} \cdot y^t) + (1,1,a_{12} \cdot z^t)
\sim
(a_{11},\I_n,y^t) + (a_{12},\I_m,z^t)
$$
belongs to $\Sigma^{-1} X_0$, 
and therefore $x \in \ker ( \mu_X )$.
\end{proof}

\newpage

\noindent
MR 2020 Subject Classification: 16S10


\bigskip

\noindent
John A. Beachy \\
Department of Mathematical Sciences \\
Northern Illinois University \\ 
DeKalb, IL 60115, USA \\
email: jbeachy@niu.edu

\end{document}